\documentclass[12pt, a4paper, draft]{amsart}
\usepackage{amssymb}
\usepackage{amsmath}

\newtheorem{theorem}{Theorem}
\newtheorem{proposition}[theorem]{Proposition}
\newtheorem{corollary}[theorem]{Corollary}
\newtheorem{lemma}[theorem]{Lemma}

\newcommand{\R}{{ \rm Re\ }}

\newcommand{\la}{\lambda}

\newcommand{\Hh}{{\mathcal H}}
\newcommand{\RR}{\mathbb{R}}
\newcommand{\CC}{\mathbb{C}}
\newcommand{\NN}{\mathbb{N}}

\newcommand{\fra}{\mathfrak{a}}

\newcommand{\frh}{\mathfrak{h}}
\newcommand{\be}{\begin{equation}}
\newcommand{\bel}{\begin{equation}\label}
\newcommand{\ee}{\end{equation}}

\begin{document}
\date{}

\title[Lieb-Thirring Estimates]{Lieb-Thirring estimates for  non self-adjoint Schr\"odinger  operators}
\author{Vincent Bruneau}
\address{Universit\'e Bordeaux 1\\
Institut de Math\'ematiques de Bordeaux, CNRS UMR 5251\\
Equipe de Math\'ematiques Appliqu\'ees\\
351, Cours de la Lib\'eration\\
33405 Talence. France.}
\email{Vincent.Bruneau@math.u-bordeaux1.fr} 

\author{El Maati Ouhabaz}
\address{Universit\'e Bordeaux 1\\
Institut de Math\'ematiques de Bordeaux, CNRS UMR 5251\\
Equipe d'Analyse et G\'eom\'etrie\\
351, Cours de la Lib\'eration\\
33405 Talence. France.}
\email{Elmaati.Ouhabaz@math.u-bordeaux1.fr}

\date{\today}

\begin{abstract}
  For general non-symmetric operators $A$, we prove that the moment of order $\gamma \ge 1$ of  negative real-parts of its eigenvalues is bounded by the moment
  of order $\gamma$ of negative eigenvalues of its symmetric part $H = \frac{1}{2} [A + A^*].$  As an application, we obtain Lieb-Thirring estimates for  non self-adjoint Schr\"odinger operators. In particular, we recover recent results by Frank, Laptev, Lieb and Seiringer  \cite{FLLS}. We also discuss moment of resonances 
of Schr\"odinger  self-adjoint operators. 
  
\end{abstract}
 
 \maketitle{} 
 
 \section{Introduction} 
 The well known Lieb-Thirring estimates   for negative eigenvalues $\la_1,...,\la_N$  of   self-adjoint Schr\"odinger operators $-\Delta + V$ say that 
 \begin{equation}\label{0}
  \sum_{k=1}^N \vert \la_k \vert^\gamma  \le L_{\gamma,d} \int_{\RR^d} V_-^{\gamma + d/2} dx.
  \ee
 Here, $V_- := \max(0, -V)$ is the negative part of $V$  and the operator $-\Delta + V$ is considered on $L^2(\RR^d, dx).$  If $\gamma = 0$, then (\ref{0}) is the well known Cwickel-Lieb-Rozenblum  estimate on the number of negative eigenvalues.  The  estimate (\ref{0})  and its analogues   are  of importance in many problems of mathematical physics. We refer the reader to Lieb \cite{Li} for a discussion and applications of (\ref{0}).   In the last years, there is an increasing interest  for non-self-adjoint Schr\"odinger operators. We refer to the review paper of Davies \cite{Da}.  Abramov, Aslanyan and Davies \cite{AAD} have proved in the one dimensional case (i.e., $d= 1$ ) that  if $\la \notin \RR^+$ is an eigenvalue of the non-self-adjoint Schr\"odinger operator $-\Delta + V$ with complex-valued potential $V,$ then 
 $$ \vert \la \vert \le \frac{1}{4} \left( \int_{\RR} \vert V \vert dx \right)^2.$$
 More recently, Frank, Laptev, Lieb and Seiringer  \cite{FLLS} proved that if $\la_1,..., \la_N$ are eigenvalues of $-\Delta + V$ (with complex-valued potential $V$) such that 
 $\R \la_j < 0$, then
 \begin{equation}\label{01}
 \sum_j ( - \R \la_j)^\gamma \le L_{\gamma, d} \int_{\RR^d}\left( \R V \right)_-^{\gamma + d/2} dx
 \ee
 for  $\gamma \ge 1.$  The constant $L_{\gamma, d}$ is the same as in (\ref{0}). This gives the analogue  of (\ref{0}) for Schr\"odinger operators with complex valued potentials. We note in  passing that (\ref{01}) for  $\gamma \in [0, 1[$ is an open question. \\
 In this note we prove in an abstract setting that if $A$ is a given operator on a Hilbert space, $H := \frac{1}{2} [A + A^*]$ is its symmetric part (see the next section for the precise definitions), and $\gamma \ge 1$ is any constant, then 
 \bel{10} 
\sum_{j} (-\R  \la_j)^\gamma \le    {\rm Tr} (H)_-^\gamma. 
\ee
 Here, ${\rm Tr} (H)_-^\gamma := \sum_j (-\mu_j)^\gamma$ where $\mu_j$ are the negative eigenvalues of the self-adjoint operator $H$  and $\la_j$ are eigenvalues of 
 $A$ with negative real-parts. If $A = -\Delta + V$, then $H = -\Delta + \Re V$ and hence using  (\ref{0}) for the self-adjoint operator  $H$ we obtain (\ref{01}). We also obtain other estimates for moments of  eigenvalues of more general differential operators $A.$ As a consequence, we obtain estimates for moments of resonances of (self-adjoint) Schr\"odinger operators. 
  
 We mention  that (\ref{10}) can be considered  as an extension to
 infinite dimension of  Ky Fan's result (see Bhatia \cite{Bha},
 p. 74). It follows from a simple variational lemma (see Lemmas \ref{le1} and
 \ref{le2}). Such lemma, well known in finite dimensional space (see
 \cite{Bha}, p. 24), is also used by Fournais-Kachmar
 \cite{FouKac}. 
 Since the  proofs are natural and
 elementary, we state this  lemma in an abstract setting and  give all the details of proof.

 \section{Moments of negative eigenvalues for non-self-adjoint  operators}
 Let $\Hh$ be a complex Hilbert space. We denote by $ \langle.,.\rangle$ its scalar product and by $\Vert . \Vert := \sqrt{ \langle.,.\rangle}$ the corresponding norm.\\
 Let $\fra$ be a densely defined continuous  and closed (non-symmetric) form on $\Hh$. We assume that $\fra$ is bounded from below. This means that there exists a constant $\eta$ such 
 $$\R \fra(u,u ) + \eta  \Vert u \Vert^2   \ge 0 \ {\rm for\ all\ } u \in D(\fra).$$
 We shall denote by  $A$ its associated operator (see for example Chapter 1 in  \cite{Ou} for the definitions). 
 Let  $\frh $ be  the symmetric part of $\fra$, that is, 
 $$ \frh (u, v) := \frac{1}{2} \left(  \fra (u,v) + \overline{\fra (v,u)}  \right)  {\rm for\  all \ †}   †  u, v \in D(\frh) = D(\fra).$$
 The form $\frh$ is symmetric and its associated operator $H$ is self-adjoint.\footnote{Formaly, one writes $H = \frac{1}{2}\left[ A + A^* \right]$.}
 Of course, $H$ is bounded from below. \\
 It is a classical fact that the spectrum $\sigma(A)$ is contained in some sector of the complex plane. In general, there is no  relationship between the spectrum 
 of the two operators $A$ and $H.$  The main result  in this section gives an estimate of the moments of  eigenvalues  of $A$  having  negative real-parts in terms of the same quantity corresponding to  $H.$  
\begin{theorem}\label{A-H}
i) If the negative spectrum of $H$ is empty, then $A$ has no eigenvalue in $\{ z \in \CC; \; \R z<0\}$.

ii) Assume now that the negative spectrum of $H$ is  discrete. 
Let $\la_1, ..., \la_N$ be any  finite family of eigenvalues of $A$ with $\R \la_j < 0$ for all $j$ (an eigenvalue with algebraic multiplicity 
$k > 1$ might occur $n$ times with $n \le k$). Let $\gamma \ge 1.$ Then
\bel{1}
\sum_{k=1}^N (-\R  \la_k)^\gamma \le    
 {\rm Tr} (H)_-^\gamma. 
\ee
\end{theorem}

Recall that the algebraic multiplicity $m_a(\la)$ of an eigenvalue $\la$ of $A$ is defined by
$$ m_a(\la) := \sup_{k \in \NN}  \dim† \ker (\la I - A)^k,$$
where dim denotes  the dimension of the  Kernel  of $(\la I - A)^k.$ 
The algebraic multiplicity  can be infinite. However,  a direct
consequence of the above  theorem is that the algebraic multiplicity
of any eigenvalue of $A$ with negative real-part is finite provided ${\rm Tr} (H)_- < \infty$. Indeed, fix an eigenvalue $\la$  with negative real-part and 
 let $k < \infty$ with $k$ less or equal to $m_a(\la).$ The estimate (\ref{1}) applied to the single eigenvalue $\la$ (with $\gamma  = 1$)  gives
 $$ k (-\R \la) \le   {\rm Tr} (H)_-.$$
 Since this holds for every finite $k \le m_a(\la)$, one  has 
 $m_a(\la) (-\R \la) \le   {\rm Tr} (H)_- < \infty.$ \\
 We also recall that for self-adjoint operators, both algebraic and geometric dimensions coincide. 
 
 One can also withdraw a similar conclusion  for eigenvalues with real-parts less than any fixed value $t$  under the condition that  the spectrum of $H$  below $t$ is discrete  (just replace $A$ by $A - tI$, $H$ is then  replaced by $H -tI$). 
 
 Another consequence is that the co-dimension of range of $ (\overline{\la} I -A)$ is finite for every  $\la$ in the residual spectrum of $A$ with $\R \la < 0.$  The obvious reason is that by duality,   $\overline{\la}$  turns to be  in the  point spectrum of its adjoint $A^*.$ The symmetric part of $A^*$ is also $H$. We then  apply  the previous observations to $A^*.$

The proof  of the above theorem is based on the following variational
lemmas for the sum of eigenvalues of a self-adjoint operator.  As
mentioned in the introduction, these are infinite dimensional versions
of Ky Fan's maximum principle \cite{Bha}. 
 \begin{lemma}\label{le1} Let $(H,D(H))$ be a self-adjoint operator in a Hilbert space $\Hh$ (with scalar product  $\langle . ,.\rangle$). Suppose the spectrum of $H$ is discrete on $]- \infty,E[$, where $E$ is any fixed real number. 

Let $E_1 \leq E_2 \leq  \cdots \leq E_N$ be the  first  $N$ eigenvalues of $H$ (in $]- \infty,E[$) repeated according to  their multiplicities (eigenvalues that are  equal to $E_N$ are repeated $k$ times with $k \leq m_a(E_N)$).
 
Then we have:
\bel{sum}
\sum_{n=1}^N E_n  = \inf \sum_{n=1}^N \langle Hu_n,u_n\rangle
\ee
where the infimum is taken over all orthonormal sets  $\{u_1, \cdots, u_N \} \subset  D(H).$ 
\end{lemma}
\begin{proof}
Let  $\{e_j\}_{1\leq j \leq N}$ be an orthonormal family of eigenvectors associated to the eigenvalues $\{E_j\}_{1\leq j \leq N}$. 
Then for any orthonormal family $\{u_k\}_{1\leq k \leq N}$ we have:
\bel{rien}
u_k= \sum_{j=1}^N \langle u_k,e_j \rangle e_j+ r_k,
\ee
with $r_k  \in  Vect \{e_1,...,e_N\}^\perp$
(the orthogonal subspace to $e_1,...,e_N$). 
Due to the orthonormal properties, we have:
\bel{pythagore}
1= \| u_k \|^2 = \sum_{j=1}^N |\langle u_k,e_j \rangle|^2 + \| r_k \|^2 ,
\ee
and
\bel{pythagoreH}
\langle Hu_k,u_k \rangle= \sum_{j=1}^N E_j \, |\langle u_k,e_j \rangle|^2 + \langle Hr_k,r_k \rangle .
\ee
By the spectral theorem (applied to the part of $H$ on $Vect \{e_1,...,e_N\}^\perp$),  we have:
$\langle Hr_k,r_k \rangle \geq \Sigma_N \| r_k \|^2$ with 
$$\Sigma_N \ge E_N.$$ 
The value $\Sigma_N$ is just the bottom of the spectrum of the part of $H$ on  $Vect \{e_1,...,e_N\}^\perp.$
 Therefore,  
\bel{m1}
\sum_{k=1}^N \langle Hu_k,u_k\rangle \geq \sum_{j=1}^N E_j \, \sum_{k=1}^N |\langle u_k,e_j \rangle|^2 + \Sigma_N \sum_{k=1}^N\| r_k \|^2.
\ee
According to (\ref{pythagore}), we have:
\bel{e1}
\sum_{k=1}^N\| r_k \|^2 = \sum_{k=1}^N \Big( 1 - \sum_{j=1}^N |\langle u_k,e_j \rangle|^2\Big)
=\sum_{j=1}^N \Big( 1 - \sum_{k=1}^N |\langle u_k,e_j \rangle|^2\Big).
\ee
On the other hand, exploiting the fact that $\{u_k\}$ is an orthonormal family and that $\|e_j\|=1$, we have
$$ 1 - \sum_{k=1}^N |\langle u_k,e_j \rangle|^2 = \|e_j\|^2 - \sum_{k=1}^N |\langle u_k,e_j \rangle|^2 \geq 
0.$$
Consequently, combining (\ref{m1}) with  (\ref{e1}), and using that for any $1 \leq j \leq N$, $\Sigma_N \geq E_j$, we obtain:
\begin{eqnarray*}
\sum_{k=1}^N \langle Hu_k,u_k\rangle  &\geq&  \sum_{j=1}^N E_j \,  \sum_{k=1}^N|\langle u_k,e_j \rangle|^2 +
\sum_{j=1}^N E_j \Big( 1 - \sum_{k=1}^N |\langle u_k,e_j \rangle|^2\Big)\\
& =&  \sum_{j=1}^N E_j .
\end{eqnarray*}
Clearly, the equality holds for $(u_1, \cdots, u_N) = (e_1, \cdots, e_N) $ and (\ref{sum}) follows.
\end{proof}
The following lemma is a variation of the previous one. Suppose we have  $N$ vectors $u_j$ and $N' $ eigenvalues  $E_1\le E_2\le ... \le E_{N'}$ (with $N' \le N$). We denote again by $e_1,...,e_{N'}$ the corresponding orthonormal family  of corresponding eigenvectors.  Let $\Sigma_{N'}$ be  the infimum of the spectrum of the part of $H$ on $Vect \{e_1,...,e_{N'}\}^\perp.$ Of course 
$$\Sigma_{N'} \ge E_{N'}.$$
Now we have
 \begin{lemma}\label{le2} 
 Let $H$ be a self-adjoint operator in $\Hh$ and $E_1 \leq E_2 \leq  \cdots \leq E_{N'}$ be the $N'$ first eigenvalues of $H$. Let ${u_1,...,u_N}$ be an orthonormal set such that $u_k \in D(H)$ for each $k.$ 
 Then, for $N\geq N'$ we have
$$\sum_{k = 1}^N  \langle Hu_k,u_k\rangle  \ge E_1 +...+E_{N'} + (N-N') \Sigma_{N'}.$$
\end{lemma}
The   proof is similar to that of the previous lemma.   It suffices to replace (\ref{rien}) by
$$ u_k= \sum_{j=1}^{N'} \langle u_k,e_j \rangle e_j+ r_k,$$
and argue as before.
 
\noindent {\bf Remark.} In the previous lemmas, it is not necessary to have $u_k \in D(H)$. It is enough to have $u_k$ in the domain of the quadratic form of $H$ (the later coincides with
$D(\sqrt{H})$). In that case $ \langle Hu_k, u_k\rangle$ is understood in the quadratic form sense and equals  $ \langle \sqrt{H}u_k, \sqrt{H}u_k \rangle.$ \\

\begin{proof} [Proof of Theorem \ref{A-H}] 
i) If the negative spectrum of $H$ is empty, then $H$ is a non negative operator. Consequently each eigenvalue of $A$, $\la_j$ associated to an eigenvector $u_j$ satisfies: 
$$\R  \la_j = \R  \langle A u_j, u_j \rangle =  \langle H u_j , u_j \rangle \geq 0.$$
ii)  Assume now that 
$\la_1,...,\la_N$ are eigenvalues of $A$ such that $\R \la_j < 0$ for $j=1,...,N.$

1) We consider first the case $\gamma = 1.$ Following   \cite{FLLS},  if  $\la_j$ has algebraic multiplicity $k$, then by the upper triangular representation one finds an orthonormal family $v_1,...,v_k$ such that
$$ A v_l = \la_j v_l + \sum_{k<l} \alpha_{kl} v_k.$$
Using this for each $\la_j$,  we obtain  an orthonormal family $\{u_1, ..., u_N \}$ which satisfies the above property. Taking the scalar product with $u_j$ yields 
\bel{f1}
 \langle A u_j, u_j \rangle = \la_j  \  {\rm  for†} †\ j= 1,...,N.
 \ee
Taking the sum, we obtain
\bel{f2}
\sum_{n=1}^N \R \la_n = \sum_{n=1}^N  \R \langle A u_n , u_n \rangle = \sum_{n=1}^N   \langle H u_n , u_n \rangle.
\ee
Here $\langle H u_n , u_n \rangle $ is understood  in the quadratic form sense because it is not clear whether $u_n \in D(H)$ (however $u_n \in D(\fra) = D(\frh)$). 
Now we apply the previous lemmas. For this, we proceed in two steps.\\
Assume that $H$ has only a finite number of negative eigenvalues. Denote these negative eigenvalues 
 by $\mu_1 \le ... \le \mu_M$ (all them are  repeated according to their multiplicities). 
If $M \le N$, we apply Lemma \ref{le2} with $N'=M.$  Note that  $\Sigma_M\geq 0$ because $H$ has only eigenvalues below $0$  (by assumption) and these eigenvalues are  $\mu_1 \le ... \le \mu_M$. Hence 
\begin{eqnarray*}
\sum_{n=1}^N \R \la_n & =&  \sum_{n=1}^N   \langle H u_n , u_n \rangle\\
&\ge& \mu_1+ ... + \mu_M\\
&=& - {\rm Tr}(H)_-.
\end{eqnarray*}
If  $M > N$,   we apply Lemma \ref{le2} with $N'=N$  and obtain
\begin{eqnarray*}
  \sum_{n=1}^N   \langle H u_n , u_n \rangle   
  &\ge&  \mu_1+ ... + \mu_{N-1} + \mu_N\\
  &\ge&  \mu_1+ ... + \mu_{N-1} + \mu_N + \mu_{N+1} + ... + \mu_M\\
  &=& - {\rm Tr}(H)_-.
  \end{eqnarray*}
  Assume now that $H$ has infinite number of eigenvalues below $0.$ We choose  $M > N$ and let $\mu_1 \le ... \le \mu_M$ be the $M$ first eigenvalues. Now we proceed as above by applying Lemma \ref{le2} with $N'=N$ and obtain
  $$ \sum_{n=1}^N   \langle H u_n , u_n \rangle \ge \mu_1+...+ \mu_N \ge - {\rm Tr}(H)_-.$$
Thus, in all cases, we have
\bel{f3}
\sum_{n=1}^N  \R \la_n \ge - {\rm Tr}(H)_-,
\ee
which proves the theorem when $\gamma =1.$\\

2) The case $\gamma > 1$ follows from the previous case. The proof uses an idea of Aizenman-Lieb \cite{AL}. It was also used in \cite{FLLS}. We follow the same 
arguments as in the proof of Lemma 1 of the later references. 
Denote again  $f_- := \max(- f, 0)$.  There exists a constant $C_\gamma$ such that 
$$ C_\gamma s_-^\gamma = \int_0^\infty t^{\gamma -2} (s+t)_-dt.$$
Therefore, applying the case $\gamma = 1$ to the operators $A + t I$ and $H + tI$, we obtain
\begin{eqnarray*}
C_\gamma \sum_{n =1}^N (-\R \la_n)^\gamma &=&  \int_0^\infty t^{\gamma -2} \sum_{n =1}^N (\R \la_n + t)_ - dt\\
& \le & -  \int_0^\infty t^{\gamma -2} {\rm Tr}(H+tI)_- dt\\
&=& C_\gamma {\rm Tr}(H)_-^\gamma.
\end{eqnarray*}
This proves the theorem.
\end{proof}

In general settings the above result does not hold for $\gamma < 1.$ The following elementary example was communicated to us by  J.F. Bony.  We thank him for 
fruitful discussion. \\

Consider on  $H := \ell_2(\NN)$  the finite rank operator $A$ defined as follows. Fix $n$ large enough and  let 
\[
Au(j) := \left\{
\begin{array}{ll} 
 - u(j)- 2u(j+1)- ...- 2u(n),† &   †1 \le j < n\\
 - u(n), &   †j = n\\
 0,  &    j > n
 \end{array}
 \right.
\]
Clearly, $-1$ is the only non zero eigenvalue of $A$ and its algebraic multiplicity is $n.$ The symmetric part $H$ of $A$ has only $-n$ as a negative eigenvalue (it has multiplicity $1$). Therefore,
$\sum_{j=1}^n 1^\gamma = n$,  Tr$(H)_-^\gamma = n^\gamma$ and  the inequality
$n \le C  n^\gamma$ cannot hold for any constant $C$  (independently of $n$ when  $\gamma < 1$).

\section{Application to non self-adjoint Schr\"odinger operators}
In this section we consider  non self-adjoint Schr\"odinger  operators 
$$P:=-\Delta + i \Big( a(x).\nabla+ \nabla.a(x)\Big) + V$$
 where $V$ is a complex-valued potential and $a$ is a complex-valued vector field. \\
We first describe how $P$ is defined. Consider the symmetric  non-negative sesquilinear form
$${\mathfrak{e}}_0(u,v) :=  \int_{\RR^d} \left(i\nabla + \R a \right)u \overline{ \left(i\nabla + \R a \right) v}  dx + \int_{\RR^d} (\R V)_+ u \overline{v} dx,$$
defined on the space $C_c^\infty$ of $C^\infty$ functions with compact support. For $\R a \in L^2_{loc}$ and $(\R V)_+ \in L^1_{loc}$, this form is closable 
(see \cite{Sim78}). We denote again by $\mathfrak{e}_0$ its closure.\footnote{the operator associated with this form is the magnetic Schr\"odinger operator
$ \left(i\nabla + \R a \right)^2 + (\R V)_+.$}
Consider now the (non-symmetric) sesquilinear form $\mathfrak{e} := \mathfrak{e}_0 + \mathfrak{e}_1$ where 
\begin{eqnarray*}
\mathfrak{e}_1(u,v) :=&& \int_{\RR^d} \left\{ {\rm Im\ } a(x) u \nabla \overline{v}  - {\rm Im\ }  a(x) \nabla u \overline{v} \right\} dx +\\ 
&&\int_{\RR^d} \left\{ - \vert \R a(x) \vert^2 - (\R V)_- + i {\rm Im\ }  V \right\} u \overline{v}dx.
\end{eqnarray*}
We assume that there exist two constants $\beta \in [0, 1[$ and $c_\beta \in \RR$ such that for every $u \in D(\mathfrak{e}_0)$
\begin{equation}\label{pert}
 \vert \mathfrak{e}_1(u,u) \vert \le \beta \mathfrak{e}_0(u,u) + c_\beta \int_{\RR^d} \vert u \vert^2 dx
 \ee
(in particular, this holds if  ${\rm Im\ }  a, (\R V)_-, {\rm Im\ }  V \in L^\infty(\RR^d)$). By the well known KLMN theorem (see for example \cite{Ka} p. 320  or \cite{Ou}, p. 12), the form 
$\mathfrak{e}$, with domain $D(\mathfrak{e}) = D(\mathfrak{e}_0)$, is well defined as is closed. One can then associate with $\mathfrak{e}$ an operator. Formally, this operator is given by $P$ above.  In the sequel, we assume that $\R a \in L^2_{loc}$, $(\R V)_+ \in L^1_{loc}$ and (\ref{pert}) are satisfied and $P$ will be the associated operator with $\mathfrak{e}.$

We want to study   Lieb-Thirring estimates  for $P$ 
(for moments of eigenvalues  having negative real-parts). The next  theorem   was  proved recently in \cite{FLLS} for potential perturbations. Our proof is easier and applies in  many situations (we can  consider for example operators with boundary conditions on domains or Schr\"odinger operators  on some manifolds). 

For real-valued potentials $W$ and real vector field $b$, the well-known Lieb-Thirring estimates for negative eigenvalues $\mu_j $ of
Schr\"odinger (self-adjoint) operators  $(i\nabla +b)^2 + W$  say that 
\bel{LTR} 
  {\rm Tr}((i\nabla +b)^2 + W)_-^\gamma := \sum_{\mu_j < 0} (-\la_j)^\gamma   \le L_{\gamma,d} \int_{\RR^d} W(x)_-^{\gamma + d/2} dx,
\ee
where $L_{\gamma, d}$ is a positive constant. See \cite{Si}, \cite{LT} for details.  Here $\gamma \ge 1/2$  if $d =1, \gamma > 0$ if $d=2$ and $\gamma \ge 0$ for $d \ge 3.$  For information concerning the   best value of the constant $L_{\gamma,d}$ see \cite{LW}, \cite{DLL}.  In the sequel, we shall refer to $L_{\gamma,d}$ as the best possible constant for which (\ref{LTR}) holds.  
For complex-valued potentials $V$ and  complex-valued vector fields $a,$ we introduce the following family of self-adjoint Schr\"odinger  operators:
\begin{equation}\label{defHalpha}
H(\alpha) := (i\nabla + b(\alpha))^2 -  |b(\alpha)|^2 + W(\alpha); \quad \alpha \in [-\frac{\pi}{2},\frac{\pi}{2}]\setminus\{0\} 
\end{equation}
with 
$$b(\alpha):= \frac{1}{ \sin \alpha}  \R(e^{-i(\alpha - \frac{\pi}{2})} a)= \R a - (\cot \alpha) {\rm Im†} a$$ 
and 
$$W(\alpha):= \frac{1}{ \sin \alpha}  \R(e^{-i(\alpha - \frac{\pi}{2})} V)=\R V -(\cot \alpha) {\rm Im†} V.$$
   We have
\begin{theorem}\label{LT} 
i) If $\alpha \in ]0,\frac{\pi}2]$ (respectively, $-\alpha \in ]0,\frac{\pi}2]$) is such that $ H(\alpha)$ is non negative, then
$$\sigma (P) \subset  e^{i[\alpha - \pi, \alpha]}{\RR}^+ \quad  ({\rm respectively,} \,† \sigma (P) \subset  - e^{i[\alpha - \pi, \alpha]}{\RR}^+).$$

ii) Fix $ \alpha \in ]0,\frac{\pi}2] $  (respectively,  $-\alpha \in ]0,\frac{\pi}2]$)  such that 
$$\Big(W(|\alpha|) - |b(|\alpha|)|^2\Big)_-  \in L^{\gamma + d/2}(\RR^d).$$
Let $\la_1, ..., \la_N$ be any  finite family of eigenvalues of $P$ contained outside the sector  $  e^{i[\alpha - \pi, \alpha]}{\RR}^+$ (respectively,  $- e^{i[\alpha - \pi, \alpha]}{\RR}^+$)  for all $j$ (an eigenvalue with algebraic multiplicity $k > 1$ might occur $n$ times with $n \le k$). 
 Then for $\gamma \ge 1,$ we have
$$\sum_{k=1}^N (-\R  \la_k + ( \cot \alpha) {\rm Im†} \la_k)^\gamma  \le L_{\gamma, d} \int_{\RR^d} \left[W(|\alpha|) - |b(|\alpha|)|^2\right]_-^{\gamma + d/2} dx,
$$
where $L_{\gamma, d}$ is the best possible value for which (\ref{LTR}) holds.
\end{theorem}
\begin{proof} We apply Theorem \ref{A-H} to  the operator  $A :=\pm e^{-i( \alpha- \frac{\pi}{2})}\, P $ with  $\pm \alpha \in ]0,\frac{\pi}{2}] $ fixed. The real part of $A$ (in the sense of quadratic forms, see the beginning  of the previous section) is 
\begin{eqnarray}\nonumber
\R A &= &\pm (\sin \alpha \; \R P - \cos \alpha \;  {\rm Im†} P)\\\nonumber
&=&-|\sin \alpha| \Delta + i|\sin \alpha| \Big(\R a(x).\nabla+ \nabla.\R a(x)\Big) +|\sin \alpha| \R V \\\nonumber
 & & \mp i(\cos \alpha) \Big({\rm Im} a(x).\nabla+ \nabla.{\rm Im} a(x)\Big) \mp  (\cos \alpha) {\rm Im}   V.\\\nonumber
&=&  |\sin \alpha| \, H( \alpha).\nonumber
\end{eqnarray}
 Observing that eigenvalues of $A$ are eigenvalues of $P$  times   $\pm e^{-i( \alpha- \frac{\pi}{2})}$, and that 
$$\R \Big(\pm e^{-i( \alpha- \frac{\pi}{2})}\la \Big) \geq 0 \Leftrightarrow \la \in \pm e^{i[\alpha - \pi, \alpha]}{\RR}^+ ,$$ 
we deduce Theorem \ref{LT}  from Theorem \ref{A-H} and  (\ref{LTR}). In fact, for the proof of ii),  we have:
\begin{equation}\label{LTalpha}
\sum_k \Big(\R(\pm e^{-i( \alpha- \frac{\pi}{2})}\la_k) \Big)_-^\gamma \le |\sin \alpha |^{\gamma} {\rm Tr}(H(\alpha))_-^\gamma.
\end{equation}
Then, dividing both sides by $|\sin \alpha |^\gamma$, yields:
$$ \sum_k (\R  \la_k - ( \cot \alpha) {\rm Im†} \la_k)_-^\gamma   \le  {\rm Tr}(H(\alpha))_-^\gamma.$$
and  Theorem \ref{LT} follows then  from (\ref{LTR}).

\end{proof}

 We also have  the following estimates for $\sum \vert \la_k \vert^\gamma.$ 

\begin{corollary}\label{LTav}
Under the assumptions and  notation of Theorem \ref{LT} ii), we have  for $\gamma\geq 1$,  $\varepsilon \in ]0, \frac{\pi}{2}]$ and $\alpha \in ]0, \frac{\pi}{2}]$:
$$
\sum_{\la_k \in   e^{i[ \alpha + \varepsilon, \alpha - \varepsilon + \pi]}{\RR}^+}\!\!\! \!\!\!\!\!\!\!\!\! | \la_k|^\gamma \le \Big( \frac{|\sin \alpha|}{\sin \varepsilon}\Big)^\gamma  L_{\gamma, d} \int_{\RR^d} \left[W(|\alpha|) - |b(|\alpha|)|^2\right]_-^{\gamma + d/2} dx.$$
If $\alpha \in [- \frac{\pi}{2}, 0[,$ then the  same estimate holds for the sum over  $\la_k \in -e^{i[ \alpha + \varepsilon, \alpha - \varepsilon + \pi]}{\RR}^+.$ 
\end{corollary}

Note that for $ \alpha \in ]0,\frac{\pi}2] $ such that 
$\Big(W(\alpha) - |b(\alpha)|^2\Big)_- \in L^{\gamma + d/2}(\RR^d)$,
the sum of the above estimates for $\alpha$  and for $-\alpha$ yields estimate for eigenvalues outside of  the sector $e^{i[ -\alpha - \varepsilon, \alpha + \varepsilon]}{\RR}^+$:
$$\sum_{\la_k \in   e^{i[ \alpha + \varepsilon, 2\pi -\alpha - \varepsilon]}{\RR}^+}\!\!\! \!\!\!\!\!\!\!\!\! | \la_k|^\gamma \le 2 \Big( \frac{|\sin \alpha|}{\sin \varepsilon}\Big)^\gamma  L_{\gamma, d} \int_{\RR^d} \left[W(|\alpha|) - |b(|\alpha|)|^2\right]_-^{\gamma + d/2} dx.$$

\begin{proof}
Assume that $\alpha \in ]0, \frac{\pi}{2}]$. Let $\alpha_k \in [\alpha  + \varepsilon, \alpha - \varepsilon + \pi ] $ 
be the argument of the eigenvalue $\la_k \in  \pm e^{i[ \alpha + \varepsilon, \alpha - \varepsilon + \pi]}{\RR}^+$. Since $\la_k = e^{i \alpha_k}\, |\la_k|$, then 
$$\R(- e^{-i( \alpha- \frac{\pi}{2})}\la_k)= |\la_k|\,  \R(- e^{-i( \alpha-\alpha_k- \frac{\pi}{2})})= - |\la_k| \, \sin( \alpha-\alpha_k) .$$
Therefore,
$$- \R( e^{-i( \alpha- \frac{\pi}{2})}\la_k) \geq  |\la_k|\, \sin \varepsilon.$$
Inserting this inequality in (\ref{LTalpha}), we obtain
$$\sum_{\la_k \in   e^{i[\alpha + \varepsilon, \alpha - \varepsilon + \pi ]}{\RR}^+} | \la_k|^\gamma \le \Big( \frac{|\sin \alpha|}{\sin \varepsilon}\Big)^\gamma {\rm Tr}(H(\alpha))_-^\gamma.$$
The assertion in Corollary \ref{LTav} follows then from  (\ref{LTR}).\\
The arguments are similar if $\alpha \in [-\frac{\pi}{2}, 0[.$
\end{proof}
Of course,  taking only one eigenvalue $\la \in e^{i[ \alpha + \varepsilon, \alpha - \varepsilon + \pi]}{\RR}^+$ of $P$, one has 
$$\vert \la \vert \le \Big( \frac{|\sin \alpha|}{\sin \varepsilon}\Big)^\gamma  L_{1, d} \int_{\RR^d} \left[W(|\alpha|) - |b(|\alpha|)|^2\right]_-^{1 + d/2} dx.$$
For  $P = -\Delta + V,$ a  result with a better constant is obtained in \cite{AAD} in the case of dimension $d = 1.$ Indeed, Theorem 4 in \cite{AAD} says in this case that 
$$\vert \la \vert \le \frac{1}{4} \left( \int_{\RR} \vert V(x) \vert dx \right)^2.$$

The previous results  can be applied   to estimate  the moments of  resonances for Schr\"odinger operators. 
Consider a potential $V \in C^\infty(\RR^d,\RR)$ such that  for every $x \in \RR^d,$  $z \to V(zx)$ has an analytic  extension 
to a  neighbourhood  of  the sector of angle 
$\theta/2 $ (for a  fixed $\theta  \in ]0, \pi]$).  Assume also that this analytic extension is relatively compact with respect to $-\Delta.$
 Then,   the resonances of $-h^2 \Delta + V(x)$ ($h>0$) in
 the sector $S_\theta := e^{i]-\theta, 0]} \RR^+$, $\theta \in [0,\pi]$ are the eigenvalues of the non-self-adjoint operator
$$P_\theta := - e^{-i\theta} \, h^2 \Delta + V(e^{i\theta/2}x).$$
See for instance \cite{CyFr}, \cite{HiSi} for a general introduction to the theory of resonances.  

First, observe that  
$$\sigma(P_\theta) = h^{2} e^{-i\theta} \, \sigma (h^{-2} e^{i\theta} P_\theta) =  h^{2} e^{-i\theta} \, \sigma (-\Delta + M),  $$
where $M(x) = h^{-2} e^{i\theta} V(e^{i\theta/2}x).$   It follows from assertion i) of Theorem \ref{LT} (which we apply  to $-\Delta + M$) that if the operator 
$$H_\theta(\alpha) := - \Delta + \frac{h^{-2}}{\sin \alpha} \R \Big( e^{-i( \alpha-\theta- \frac{\pi}{2})}V(e^{i\theta/2}\,.\,)\Big)$$
is non-negative, then there are no resonances in the sector $S_{\alpha-\theta}.$ 
Related results on localization  of resonances are studied in several
settings. Localisation results which  relay  on numerical range are given in  \cite{AAD}.
 A lot of results, often related to some dynamical assumptions, are obtained by microlocal 
arguments  (see for instance \cite{Bu}, \cite{SjZw}, \cite{Ma} and
\cite{NoZw}). For the magnetic Schr\"odinger operator, we
also mention  \cite{BoBrRa} where localization of resonances is obtained by perturbation methods.

If $H_\theta(\alpha)$ is not non-negative, we can apply Corollary \ref{LTav}  to obtain an estimate for moments of resonances. Let 
$w_j \in \sigma(P_\theta) \cap S_\theta$ be resonances  of  $-h^2 \Delta + V(x)$. We apply  Corollary \ref{LTav} (to $-\Delta  +M$) and obtain 
$$
\sum_{w_j \in e^{i[ \alpha -\theta + \varepsilon, 0]}{\RR}^+}\!\!\! \!\!\!\!\!\!\!\!\! | h^{-2} e^{i\theta} w_j |^\gamma \le $$
$$\Big( \frac{|\sin \alpha|}{\sin \varepsilon}\Big)^\gamma  L_{\gamma, d} \int_{\RR^d} \left[\frac{h^{-2}}{\sin \alpha} \R \Big( e^{-i( \alpha-\theta- \frac{\pi}{2})}V(e^{i\theta/2}x)\Big)\right]_-^{\gamma + d/2} dx.$$
Consequently, for any $h>0$, $\alpha \in ]0,\theta[$, $\varepsilon \in ]0, \theta-\alpha[$, we have:
 
\begin{eqnarray*}
\sum_{w_j \in e^{i[ \alpha -\theta + \varepsilon, 0]}{\RR}^+}\!\!\! \!\!\!\!\!\!\!\!\! |  w_j |^\gamma\!\!\! &\le& \, \frac{h^{-d} L_{\gamma, d}}{(\sin \varepsilon)^\gamma\,  (\sin \alpha )^{d/2 } }
\int_{\RR^d} \left[ {\rm Im}\Big(e^{i( \theta-\alpha)}V(e^{i\theta/2}x)  \Big) \right]_+^{\gamma + d/2} dx.
\end{eqnarray*}

In particular, we have proved the following proposition. 
\begin{proposition}\label{estiresonances}
Suppose that $V   \in C^\infty(\RR^d,\RR)$ and satisfies the above analyticity and relative compactness (with respect to $-\Delta$) assumptions.   
Then for $\gamma \geq 1$  and  $\varphi \in [0, \theta[$, the resonances of $-\Delta + V$ in $e^{i[ -\varphi, 0]}{\RR}^+$ satisfy:
\begin{eqnarray*}
\sum_{w_j \in e^{i[ -\varphi, 0]}{\RR}^+}\!\!\! \!\!\!\ |  w_j |^\gamma\!\!\! &\le& \, \frac{h^{-d} L_{\gamma, d}}{(\sin \varepsilon)^\gamma\,  (\sin (\theta-\varphi-\varepsilon))^{d/2 } }
\\
&&\times  
\int_{\RR^d} \left[ {\rm Im}\Big(e^{i( \varphi + \varepsilon)}V(e^{i\theta/2}x)  \Big) \right]_+^{\gamma + d/2} dx.
\end{eqnarray*}
for all $\varepsilon \in ]0, \theta-\varphi] $.
\end{proposition}

\noindent {\bf Remark.} The above estimate give the well known upper bound $O(h^{-d})$ for the number of resonances in a compact domain (see \cite{SjZw2}). Here, we have an explicit coefficient of $h^{-d}$.

On the other hand it is clear that Corollary \ref{LTav} allows to extend Proposition \ref{estiresonances} to first order perturbations.


\begin{thebibliography}{00}
 
\bibitem {AAD}  A.A. Abramov, A. Aslanyan and E.B. Davies, 
Bounds on complex eigenvalues and resonances, 
J. Phys. A 34 (2001), no. 1, 57--72. 

\bibitem{AL} M. Aizenman and E.H. Lieb, On semi-classical bounds for eigenvalues of Schr\"odinger operators, Phys. Lett. 66 A (1978) 427--429.

\bibitem{Bha} R. Bhatia, Matrix Analysis, Graduate Texts in
  Mathematics, Vol. 169, Springer 1996.

\bibitem{Bu} N. Burq, D\'ecroissance de l'\'energie locale de l'\'equation des ondes pour le probl\`eme ext\'erieur et absence de r\'esonance au voisinage du r\'eel,  Acta Math.  180  (1998),  no. 1, 1--29.

\bibitem{BoBrRa} J.F. Bony, V. Bruneau, G. Raikov, Resonances and spectral shift function near the Landau levels,  Ann. Inst. Fourier (Grenoble)  57  (2007),  no. 2, 629--671.

\bibitem{CyFr} H.L. Cycon, R.G. Froese, W. Kirsch and B. Simon, Schr\"odinger operators; with application to quantum mechanics and global geometry, Texts and Monographs in Physics. Springer Study Edition. Springer-Verlag, Berlin (1987).

\bibitem{DN} E.B. Davies and J. Nath, Schr\"odinger operators with slowly decaying potentials, J. Comput. Appl. Math. 148, no 1 (2002) 1-- 28.

\bibitem{Da} E.B. Davies, Non-self-adjoint differential operators, Bull. London Math. Soc. 34 (2002) 513-532.

\bibitem{DLL} J. Dolbeault, A. Laptev and M. Loss, Lieb-Thirring inequalities with improved constants, Preprint 2007.

\bibitem{FouKac} S. Fournais, A. Kachmar, On the energy of bound states for magnetic Schrödinger operators, Preprint 2008.

\bibitem{FLLS} R. Frank, A.  Laptev, E.H. Lieb, and R.  Seiringer,  Lieb-Thirring inequalities for Schr\"odinger operators with complex-valued potentials.  Lett. Math. Phys.  77  (2006),  no. 3, 309--316. 

\bibitem{HiSi} P.D. Hislop and I.M. Sigal, Introduction to Spectral Theory. With applications to Schr\"odinger operators. Applied Mathematical Sciences, 113. Springer-Verlag, New York, 1996.

\bibitem{Ka} T. Kato, Perturbation Theory For Linear Operators, 2nd Edition, Springer-Verlag 1980. 

\bibitem{LW} A. Laptev and T.  Weidl, Sharp Lieb-Thirring inequalities in high dimensions.  Acta Math.  184  (2000),  no. 1, 87--111.

\bibitem{LT} E.H. Lieb and W. Thirring, Inequalities for moments of  eigenvalues of the Schr\"odinger  Hamiltonian and their relation to Sobolev inqualities. Studies in Math. Phys.
Essays in honour of  Valentine Bargmann, Princeton  (1976) 269--303.

\bibitem{Li} E.H. Lieb, Kinetic energy bounds and their application to the  stability of matters, Proc. Nordic Summer School in Math. Lecture Notes in Physics 354 (1989) 371--382.

\bibitem{Ma} A. Martinez, Resonance free domain for non globally analytic potentials,  Ann. Henri Poincar\'e  3  (2002),  no. 4, 739--756.
     
\bibitem{NoZw} S. Nonnenmacher and M. Zworski, Distribution of
  resonances for open quantum maps,  Comm. Math. Phys.  269  (2007), no. 2, 311--365.

\bibitem{Ou} E.M. Ouhabaz, Analysis Of Heat Equations On Domains, London Math. Soc. Monographs 32, Princeton Univ. Press 2005. 

\bibitem{Sim78} B. Simon, Maximal and minimal Schr\"odinger forms, J. operator Theory 1 (1979) 37-47.

\bibitem{Si} B. Simon, Trace Ideals and Their Applications, Cambridge
  Univ. Press, Cambridge 1979. 

\bibitem{SjZw} J. Sj\"ostrand and M. Zworski, Asymptotic distribution
  of resonances for convex obstacles, Acta Math.  183  (1999),  no. 2,
  191--253.

\bibitem{SjZw2} J. Sj\"ostrand and M. Zworski, Complex scaling and the
  distribution of scattering poles,  J. Amer. Math. Soc.  4  (1991),
  no. 4, 729--769.

\end{thebibliography}
\end{document}